\documentclass[12pt, a4paper]{amsart}
\usepackage[dvips]{epsfig}
\usepackage{datetime}
\usepackage[strings]{underscore}
\usepackage{amsmath}
\usepackage{amssymb}
\usepackage{amsfonts}
\usepackage{dsfont}
\usepackage{bbm}
\usepackage{amsthm}
\usepackage{amsbsy}
\usepackage{amsgen}
\usepackage{amscd}
\usepackage{amsopn}
\usepackage{amstext}
\usepackage{amsxtra}
\usepackage[utf8]{inputenc}
\usepackage{enumerate}
\usepackage{hyperref}
\usepackage{lineno}
\usepackage{marginnote}
\usepackage{verbatim}
\usepackage{calrsfs}
\usepackage{times, graphicx}
\usepackage{eso-pic}
\usepackage{lastpage}
\DeclareMathAlphabet{\pazocal}{OMS}{zplm}{m}{n}
\usepackage[foot]{amsaddr}

\numberwithin{equation}{section}

\newtheorem{theorem}{Theorem}[section]
\newtheorem{proposition}[theorem]{Proposition}
\newtheorem{lemma}[theorem]{Lemma}
\newtheorem{corollary}[theorem]{Corollary}

\theoremstyle{definition}
\newtheorem{definition}[theorem]{Definition}

\newtheorem{remark}[theorem]{Remark}
\newtheorem{conjecture}[theorem]{Conjecture}

\newcommand{\ord}{\operatorname{ord}}

\newcommand {\Z} {\mathbb{Z}}
\newcommand {\N} {\mathbb{N}}

\newcommand{\Ponepow}{\widehat{\mathbb{P}_{1}}}
\newcommand{\Pzero}{\mathbb{P}_{0}}

\begin{document}

\title[Automatic sequences]{A note on multiplicative automatic
  sequences, II}

\author{Oleksiy Klurman}
\email{lklurman@gmail.com}
\author{Pär Kurlberg}
\email{kurlberg@math.kth.se}\address{Department of Mathematics, KTH Royal Institute of Technology, Stockholm}

\date{\today,  \currenttime}
\begin{abstract}
  We prove that any $q$-automatic multiplicative 
function $f:\mathbb{N}\to\mathbb{C}$ either essentially coincides with
 a Dirichlet
  character, or vanishes on all sufficiently large primes. This confirms a strong form of a conjecture of J. Bell, N. Bruin, and
  M. Coons.
\end{abstract}	
\maketitle

\section{Introduction}
Automatic sequences play important role in computer science and number
theory. For a detailed account of the theory and applications we refer
the reader to the classical monograph~\cite{AS}. One of the
applications of such sequences in number theory stems from a
celebrated theorem of Cobham~\cite{COB}, which asserts that in order
to show the transcendence of the power series $\sum_{n\ge 1}f(n)z^n$
it is enough to establish that the function$f:\mathbb{N}\to\mathbb{C}$
is {\it not} automatic. In this note, rather than working within the
general set up, we confine ourselves to functions with the range in
$\mathbb{C}.$ There are several equivalent definitions of automatic
(or more precisely, $q$-automatic) sequences. It will be convenient
for us to use the following one.
\begin{definition} The sequence $f:\mathbb{N}\to\mathbb{C}$ is called
  $q$-automatic if the $q$-kernel of it, defined as a set of
  subsequences  
\[K_q(f)=\left\{ \{f(q^in+r\}_{n \ge 0} \vert\ i\ge 1,0\le r\le q^i-1\right\},\]
is finite. 
\end{definition}
We remark that any $q-$automatic sequence takes only finitely many
values, since it is a function on the states of finite automata.  
A
function $f:\mathbb{N}\to\mathbb{C}$ is called multiplicative if $f(mn)=f(m)f(n)$ for all pairs $(m,n)=1.$ The
question of which multiplicative functions are $q$-automatic attracted considerable attention of several authors including~\cite{Yazd},
\cite{Puchta}, \cite{BBC}, \cite{Puchta1}, \cite{Li}, \cite{KK}  and \cite{AG}.
In particular, the following conjecture was made in \cite{BBC}.
\begin{conjecture}[Bell-Bruin-Coons]\label{Conj}
For any multiplicative $q$-automatic function
$f:\mathbb{N}\to\mathbb{C}$ there exists an eventually periodic function
$g:\mathbb{N}\to\mathbb{C},$ such that $f(p)=g(p)$ for all primes $p.$  
\end{conjecture}
Some progress towards this conjecture has
been made when $f$ is assumed to be completely multiplicative. In
particular, Schlage-Puchta~\cite{Puchta} showed that a completely
multiplicative $q$-automatic sequence which does not vanish is almost
periodic. Hu~\cite{Hu} improved on that result by showing that the
same conclusion holds under a slightly weaker hypothesis. Allouche and Goldmakher~\cite{AG} considered related question classifying what they called ``mock" Dirichlet characters.
Finally Li~\cite{Li} and the authors~\cite{KK} very recently proved
Conjecture~\ref{Conj} when $f$ is additionally assumed to be
completely multiplicative. Moreover, the methods of~\cite{KK}
settled the Conjecture~\ref{Conj} fully under the assumption
of the Generalized Riemann Hypothesis. The key tool there was a
variant of Heath-Brown's results \cite{HB} on Artin's primitive
root conjecture. 
In this article, we develop alternative, more combinatorial approach,
and prove a strong form of Conjecture~\ref{Conj}. 
\begin{theorem}\label{main}
Let $q\ge 2$ and let $f:\mathbb{N}\to\mathbb{C}$ be 
multiplicative $q$-automatic sequence. Then, there exists a Dirichlet
character $\chi$ and an integer $Q\ge 1,$ such that either $f(n)=\chi(n),$ for all
$(n,Q)=1$ or $f(p)=0$ for all sufficiently large $p.$\end{theorem}
In fact, our proof yields a more refined information about the set of
values of $f$ (cf.  Remark~\ref{general}.)
We remark that J. Koneczny, in forthcoming work (private
comunication), established
a variant of Theorem~\ref{main} using  different 
methods relying on the structure theory of automatic sequences. 
  %

\section{Proof of the main result}
\label{sec:notations}
Let $\Pzero$
the set of primes for which $f(p)=0.$
Since $f$ is
$q$-automatic, it is well-known that the image of
$f:\mathbb{N}\to\mathbb{C}$ is finite and 
therefore if we define the sets
$$\mathbb{P}_{> 1}=\{ p \text{ prime}: |f(p^e)| >1 \text{ for some
  $e \ge 1$} \}$$ and
$$\mathbb{P}_{< 1}=\{ p \text{ prime}: |f(p^e) |< 1 \text{ for some
  $e \ge 1$} \},$$
then, from multiplicativity of $f$ we easily deduce $|\mathbb{P}_{>1}|,|\mathbb{P}_{<1}|<\infty.$
We begin by assuming additionally that $|\Pzero|<\infty.$  
\begin{proposition}\label{key1}
Let $f:\mathbb{N}\to\mathbb{C}$ be a $q$-automatic 
multiplicative function and suppose that $|\Pzero|<\infty.$ 
Then $f(p^k)=\chi(p^k)$ for all $p\notin \mathbb{P}_{>1}\cup \mathbb{P}_{<1}\cup\Pzero$ and any $k\ge 1.$
\end{proposition}
\begin{proof}
  Since $f$ is $q$-automatic, we have that the kernel $K_q(f)$ is finite. By the pigeonhole principle, there exist positive integers
  $i_1\ne i_2$ such that $f(q^{i_1}n+1)=f(q^{i_2}n+1)$ for all
  $n\ge 1.$ If
  $n=m\prod_{p\in \mathbb{P}_{>1}\cup \mathbb{P}_{<1}\cup\Pzero}p$
  then
  \begin{equation}\label{correlation}
  \frac{f(q^{i_1}m\prod_{p\in \mathbb{P}_{>1}\cup
      \mathbb{P}_{<1}\cup\Pzero}p+1)}{f(q^{i_2}m\prod_{p\in
      \mathbb{P}_{>1}\cup \mathbb{P}_{<1}\cup\Pzero}p+1)}=1\ne 0,
      \end{equation}
  for all $m\ge 1.$ The conclusion now immediately follows from
  Theorem $2$ of~\cite{EK} (the result is stated only for completely
  multiplicative functions, but the proof is in fact also valid for
  multiplicative 
  functions). Alternatively, one could also apply correlation formulas developed in~\cite{Klu} to get the result.
  \end{proof} In what follows we will assume that $|\Pzero|=\infty$
and show that in this case $f(p)=0$ for all sufficiently large primes $p.$ To this end, we perform the following two reductions:
\begin{itemize}
\item replace $f$ by $g=|f|$ which is also $q$-automatic
\item  change each value $|f(p^k)|\ne \{0,1\}$ to $g(p^k):=1.$ 
\end{itemize}
These two operations preserve $q-$ automaticity of the modified multiplicative sequence $\{g(n)\}_{n\ge 1}$ and do not change the zero set $\Pzero.$ It is therefore enough to prove the claim for the binary valued $f:\mathbb{N}\to \{0,1\}.$\\   
To facilitate our discussion, we introduce the set
$$
\Ponepow := \{ p \text{ prime}: f(p^e) = 1 \text{ for some
  $e \ge 1$} \}.
$$
Let $s_{0} < \infty$ denote the number of distinct sequences
$\{f(q^{i}m +r)\}_{m \in \N}$ as $i,r \geq 0$ ranges over pairs of integers
such that $r
\in [0, q^{i}).$
Given prime $p$ and $\delta\ge 1,$ we define $\alpha_{p,\delta}$ by $q^{\alpha_{p,\delta}} ||
p^{\delta \phi(q)} - 1$.  Further, given $p \in \Ponepow$, let
$$
\alpha_{p} := \min \{ \alpha_{p,\delta} : f(p^{\delta}) = 1 \},
$$
and let
$$
\delta_{p} := \min \{ \delta : \alpha_{p,\delta} = \alpha_{p}\}.
$$

We also note that if $f$ is $q$-automatic, then there exists
$k_0=k_0(f)$ with the following property: if for some $i\ge 1$ and
$0\le r\le q^i-1$ the equality $f(q^in+r)=0$ holds for all integer
$n \in [1,k_0]$, then $f(q^in+r)=0$ for all $n\ge 1.$

\subsection{Preparatory lemmas.}
\label{sec:preliminaries}
We start with a few technical lemmas. 
\begin{lemma}\label{nontriv}
Assume that $|\Ponepow|=\infty$ and
let $r>0$ be an integer such that $f(r) = 1.$ Select $A$ such that
$q^{A}>r$.  Then there exists  $i,j \in (A, A+s_{0}]$ such
that $0 < j-i  \le s_{0}$ with the property that 
$$
f( q^{i}m + r ) = f(q^{j}m+r) \quad \forall m \in \N,
$$
and $f(q^{i} m_{0} + r ) = 1$ for some integer $m_{0} \in [1, k_{0}]$.

\end{lemma}
\begin{proof}
  By our assumption, there
  exists an infinite subset $S \subset \Ponepow$ and
  $r_{S} \in \N$ such that $p^{e_p}\equiv r_S \mod q^{A+s_{0}}$ and $f(p^{e_p})=1$ for all
  $p\in S.$ Select
  $$R := \prod_{rq<p_i \in S,\ i\le \phi(q^{A+s_0})} p^{e_{p_i}}_i \equiv
  r_S^{\phi(q^{A+s_0})}\equiv 1 \mod q^{A + s_{0}}$$
  and let $r_{1} = r R$.  Clearly $r_{1} \equiv r \mod q^{A+s_{0}}$
  and for any integer $\l \in (A,A+s_{0}]$ we can write
  $r_{1}= q^{\l}m_{\l} +r.$ Since the sequence $\{f(n)\}_{n\ge 1}$ is
  $q-$automatic, by the pigeonhole principle there exist two indices
  $i,j \in (A,A+s_{0}]$ (say with $i < j$) such that
  $f(q^{i}m +r) = f(q^{j}m+r)$ for all $m \in \N$.  Moreover,
  $f(q^{i}m_{i}+r) = f(r_{1}) = f(r)f(R)=1.$ Since $r_1$ can be made
  arbitrary large, we see that any sequence
  $\{f(q^{i}m+r) \}_{m \in \N},$ $i\in (A,A+s_0]$ contains infinitely
  many ones. Hence there exists some $m_{0} \leq k_{0}$ such that
  $f(q^{i}m_{0} +r ) = f(q^jm_0+r)=1.$
\end{proof}

We next show that $f$ is non-vanishing along an exponentially growing
subsequence.
\begin{corollary}\label{exponent}
Assume that $|\Ponepow|=\infty.$ Given an integer  $r >
qk_0^2$ such that
$f(r) = 1$ there exist integers $A,C \geq 0$ and 
$1\le m_{0} \leq k_{0}$, with $C \le s_{0}$, with the property that 
\begin{equation}
  \label{eq:1}
f(q^{A+Cn}   m_{0} + r ) = 1 \quad \forall n \in \N.
\end{equation}
\end{corollary}
\begin{proof}
We apply Lemma~\ref{nontriv} to find $i,j\in (A,A+s_0],$ such that 
 $$
f( q^{i}m + r ) = f(q^{j}m+r) \quad \forall m \in \N.
$$
Without loss of generality we may assume that $A=i$; for convenience
set $C=j-i$, we then have $0 < C \le s_{0}.$ Note that
$f(q^{A+C}m+r) = f(q^{i} (q^{j-i}m)+r) = f(q^{j}m+r) = f(q^{i}m+r) =
f(q^{A}m+r)$
for all $m \in \N.$ An easy induction argument now shows that for any
$n \in \N$ we have $f(q^{A+Cn}m+r) = f(q^{A}m+r)$ for all $m \in \N.$
Lemma~\ref{nontriv} also yields an integer $m_{0} \in (0, k_{0}]$ such
that $f(q^{A} m_{0} + r ) = 1.$ This concludes the proof.
\end{proof}

In order to better illustrate the main idea of our proof, we first
focus on the case $q$ being prime and then point out necessary
modifications needed to treat the general case in
Section~\ref{sec:composite-q}. 

\begin{lemma}\label{increase}
  Assume that $f(q^{A+Cn} m_{0} + r) = 1$ for all $n \in \N$, and that
  $p^{\delta} || q^{A+Cn_{\delta}} m_{0} +r$ for some $n_{\delta}\ge 0$ with
  $\delta > \gamma$ where
$p^{\gamma}|| (q^{C })^{\ord_p(q^C)}-1$, and $p \nmid m_{0}$.
Then, given any $k \ge \gamma$,
  there exist $n_{k}$ such that $p^{k} || q^{A+Cn_{k}} m_{0} +r$.  In
  particular, $f(p^{k}) = 1$ for all $k \ge \gamma$, and 
  there exists $\alpha \in \N$ such that for any
  $\alpha_{1} > \alpha$, any element in
  $1+q^{\alpha}\Z_{q}/1+q^{\alpha_{1}}\Z_q$ is of the form $p^{k}$ for
  some (arbitrarily large) $k \in \N$.
{(With $\Z_q$ denoting the $q$-adic integers, we identify
 the unit subgroup $\{ u \in \Z/q^{\alpha_{1}} \Z : u \equiv 1 \mod
 q^{\alpha} \}$ with $1+q^{\alpha}\Z_{q}/1+q^{\alpha_{1}}\Z_q$.)}  
\end{lemma}
\begin{proof}
  By the lifting exponent lemma (cf. \cite{wiki}) we have that for all
  $n\ge 1,$
\[\nu_p(q^{C\ord_p(q^C)n} -1)=\nu_p(q^{C\ord_p(q^C)}-1)+\nu_p(n)=\gamma+\nu_p(n).\]
Therefore, we have 
\[q^{A+C(n_\delta+n\ord_p(q^C))}m_0+r=(q^{A+Cn_{\delta}}m_0+r)q^{nC\ord_p(q^C)}-r(q^{nC\ord_p(q^C)}-1).\]
Note that for any $0\le\alpha<\delta-\gamma$ we therefore have 
\[p^{\alpha+\gamma}\vert| q^{A+C(n_\delta+p^{\alpha}\ord_p(q^C))}m_0+r\]
and \[p^{\delta}\vert| q^{A+C(n_\delta+p^{\delta-\gamma+1}\ord_p(q^C))}m_0+r_0.\]
Select $n_t=tp^{\delta-\gamma}$ for some $1\le t<p$ and write   
$q^{p^{\delta-\gamma}C\ord_p(q^C)}=1+p^\delta B$ with $(B,p)=1.$ By the binomial theorem, we have
$$q^{tp^{\delta-\gamma}C\ord_p(q^C)}=(1+p^\delta B)^t=1+tp^{\delta}B+p^{2\delta}D.$$
Consequently,
\[q^{A+C(n_\delta+n_t\ord_p(q^C))}m_0+r\equiv (q^{A+Cn_{\delta}}m_0+r)(1+tBp^{\delta})-rtBp^{\delta}\mod q^{\delta+1}.\]
Select $t=t_{\delta+1},$ such that $$p| rtBt_{\delta+1}-\frac{(q^{A+Cn_{\delta}}m_0+r)}{p^{\delta}}$$
and note that for such defined $t_{\delta+1}$ we have
$p^{\delta+1}\vert
q^{A+C(n_\delta+n_{t_{\delta+1}}\ord_p(q^C))}m_0+r.$ Replacing now
$\delta\to\delta+1$ in the statement of the lemma and running the same
proof again the first claim follows. Moreover, by choosing appropriate
$n_k$ such that for $k\ge \gamma$ we have $p^k\vert| q^{A+Cn_k}m_0+r$,
we ensure that $f(p^k)=1.$ Finally, using the lifting exponent lemma together with the binomial theorem as before for the sequence $\{p^{n\phi(q)}\}_{n\ge 1}$ yields the last claim of the lemma. \end{proof}
\subsection{Concluding the proof.}
We are ready to prove the first key proposition.
\begin{proposition}\label{accumulation}
Let $f:\mathbb{N}\to\{0,1\}$ be multiplicative $q-$automatic
sequence and suppose that $|\Ponepow|=\infty.$ Then
 $$|\{ p \in
\Ponepow : \alpha_{p} = \alpha \}| = \infty,$$ 
for some $\alpha \in \N.$
\end{proposition}
\begin{proof}
Suppose that $|\{ p \in \Ponepow : \alpha_{p} =
\alpha \}| < \infty$  for all $\alpha \in \N$, and consequently, for
any $\beta \in \N$ we have
\begin{equation}
  \label{eq:2}
|\{ p \in \Ponepow : \alpha_{p}  \leq \beta \}| < \infty.
\end{equation}
Let $S = (\Ponepow \cap [1,k_{0}]) \setminus \{q\},$ and
given cut off parameter $\beta \in \N,$ we define the sets of ``medium" and ``large" primes as follows: 
$$
M  = M_{\beta}:= 
\{ p \in \Ponepow : (p,q) = 1, p > k_{0},\  \alpha_{p}  \le \beta \}
$$
and
$$
L  = L_{\beta}:= 
\{ p \in \Ponepow : (p,q) = 1, p > k_{0},\  \alpha_{p} > \beta \}.
$$
By (\ref{eq:2}), we find that both sets $M$ and $S$ are finite.

Let
$$
r =  \prod_{p \in S} p^{e_{p}}  \cdot
\prod_{p \in M} p^{e_{p}} 
$$
with the exponents $e_{p}$ chosen as follows. 
\begin{itemize}
\item For $p \in S$, take $e_{p}=0$ if the set of exponents
$\{ \delta : f(p^{\delta}) = 1 \}$ is finite. Otherwise chose $e_{p}$
so that $p^{e_{p}} > k_{0}$ and $f(p^{e_{p}}) = 1.$ In both cases, there exists
$D=O(1)$ with the property that
for any $p \in S$ such that 
$p^{\delta} | q^{n}m_{0} + r$ for some $1\le m_{0} \le k_{0}$ and $n$ such
that  $f(q^{n}m_{0} + r) = 1$, we have $\delta \le D$.
\item For $p \in M$, let $e_{p}=\delta_{p}$ and observe that 
  $f(p^{e_{p}})=1$. 
\end{itemize}
We first choose $\Delta$ such that
$q^{\Delta} \equiv 1 \mod \prod_{p \in S} p^{D}.$
We then select primes $p_{1}, \ldots, p_{\l}\in M$ with $\l$ to be specified later,
such that $\alpha_{p_{i}} < \alpha_{p_{i+1}}$, and so that for any
$C \in [1, k_{0}]$, 
\begin{equation}\label{shift}
q^{C \Delta} \not \equiv 1 \mod p_{i}
\end{equation} for all
$i \in [1,l]$.  Note that such primes exist if we take $\beta$
sufficiently large. Further, we can assume that $\beta$ is large enough so
that  $\alpha_{p_{l}} < \beta$.
Finally, for $i \in \{1, \ldots, l\}$ we define
$$
r_{i} := \frac{r}{p_{i}^{e_{p_{i}}}}\cdot
$$

We note that $f(r_{i})=1$  for $1\le i  \le l.$ Corollary~\ref{exponent} now implies that there exists
$A_{i},C_{i}, m_{i}$ with $C_{i} \le s_{0}$, $m_{i} \le k_{0}$, and
$f(q^{A_{i}+C_{i}n}m_{i} + r_{i}) = 1$ for all $n\in\N$.
Now, if $k_{i} = k_{i}(n) = q^{A_{i}+C_{i}n}m_{i} + r_{i}$, for
$n > \beta+1$, write
$$
k_{i} = 
\prod_{p} p^{f_{p,i}}
=
k_{i}^{S}  k_i^{M} k_i^{L}
$$
where for a subset $X$ of primes, we set $k_{i}^{X} = \prod_{p \in X} p^{f_{p,i}}.$

For each $k_{i}^{S}$ we have $k_{i}^{S} = \prod_{p \in S} p^{f_{p,i}}$
with $f_{p,i} \le D$.  Thus the set of possible $k_{i}^{S}$'s
is finite, and 
there exists $i \neq j$ such that $k_{i}^{S} = k_{j}^{S} =: k^{S}$
(provided $l$ is chosen large enough; note that $l$ only depends
on the number of possible $k_{i}^{S}$, which in turn depends on $S$,
and hence $l=O(1)$).

For such defined $i\ne j,$ by the construction $$(r_{j}/k^{S})^{\phi(q)}: (r_{i}/k^{S})^{\phi(q)}=(p_{i}^{e_{p_{i}}}/p_{j}^{e_{p_{j}}})^{\phi(q)}$$ and since
$\alpha_{i} < \alpha_{j}<\beta,$
we have that at least one of $(r_{i}/k^{S})^{\phi(q)}$,
$(r_{j}/k^{S})^{\phi(q)}$ does not belong to $1+q^{\beta+1} \Z_q.$ 
Without loss of generality we may assume
\begin{equation}
\label{eq:3}
(r_{i}/k^{S})^{\phi(q)} \not \in 1+q^{\beta+1} \Z_q.
\end{equation}
Clearly $k_{i}$ is coprime to all $p \in M \setminus \{p_{i}\}$ and thus the only possibility is 
$k_{i}^{M} = p_{i}^{e}$ for some $e \geq 0$.  We note
that  replacing $n$ by $n+\Delta$ does not change $k_i^{S}(n)=k_i^{S}(n+\Delta)=k_{i}^{S}$, whereas by the choice~\eqref{shift}
$k_{i}$ cannot remain divisible by $p_{i},$  In particular we can
chose $n$ so that $k_{i}^{S} = k^{S}$ and $k_{i}^{M}=1$.  Hence,
provided $n$ is large enough, we find that $r_{i}/k^{S} = k_{i}^{L}
\mod q^{\beta+1}$, which  contradicts (\ref{eq:3}) since
$(k_{i}^{L})^{\phi(q)} \in 1+q^{\beta+1} \Z_q$.
\end{proof}
We also need the following simple lemma.
\begin{lemma}
\label{lem:generate-small-q-patch}
  If $|\{ p \in \Ponepow : \alpha_{p} = \alpha \}| = \infty$ for some
  $\alpha \in \Z^+$, then for any $\alpha_{1} > \alpha$ all elements
  in $1+q^{\alpha}\Z_q/1+q^{\alpha_{1}}\Z_q$ can be written as finite
  products of coprime elements $p^{\delta_{p}}$ for $p \in \Ponepow$.
\end{lemma}
\begin{proof}
The proof is essentially the same as of Lemma~\ref{increase}.
\end{proof}
We can now conclude  the proof.
\begin{proposition}
\label{prop:concluding-proof}
Let $f:\mathbb{N}\to\{0,1\}$ be multiplicative $q-$automatic
sequence and $|\Pzero|=\infty.$ Then $f(p)=0$ for
sufficiently large $p.$ 
\end{proposition}
\begin{proof}
If $\Ponepow$ is finite, then to conclusion is immediate. Otherwise, Proposition~\ref{accumulation} implies that 
$|\{ p \in \Ponepow : \alpha_{p} = \alpha \}| = \infty$ for some
  $\alpha \in \Z^+.$
Given $q<p_{1}, \ldots, p_{k_{0}} \in  \Pzero$, i.e., primes such that
$f(p_{i}) = 0$, let
$$
Q := \prod_{i \le k_{0}} p_{i}.
$$
Given any $A$ such that
$q^{A} > 100Q^{2},$ select $r_{A}$ modulo $Q^2$ via the Chinese remainder theorem to satisfy
$$
r_{A} \equiv -sq^{A}+p_s \mod p^2_{s}, \quad 1 \le s \le k_{0}.
$$
For such $r_A,$ we have $ 0 < r_{A} < Q^{2} < q^{A}$ and  
 $p_{s} || r_{A} + q^{A}s.$ Hence $f(q^As+r_A)=0$
for $1\le s \le k_{0}.$ Moreover, for $m <
q^{A}/Q^{2} - 1$, we have
$$
r_{A,m} := r_{A} + m Q^{2} < q^{A}.
$$
Thus, as $f(q^{A}n + r_{A,M}) = 0$ for $1\le n \le k_{0},$ we find that 
in fact $f(q^{A}n + r_{A,M}) = 0$ for all $n\ge 1.$
Define the set $$N_{A} :=\{ r_{A,m} : m < q^{A}/Q^{2} -1\} \subset
\Z/q^{A}\Z,$$ and note that 
$m\mod q^{A} \in N_A$ implies that $f(m) = 0.$
%
%
We next show that $N_{A}$ has a certain multiplicative invariance
property: let $t$ be any prime such that $f(t) =1$.  Then, if
$n \in \Z$ has the property that $t \nmid n $ and
$t n \mod q^{A} \in N_{A}$, we find that
$$
0 = f(tn) = f(t) f(n)
$$
which forces $f(n)=0$.  In other words, if
$n \mod q^{A} \in t^{-1} N_{A}$, then $f(n)=0,$ provided that
$t \nmid n.$  

Now choose $e \in \N$ such that $q^{e} > Q^{2}$ and $A$ so that
$A-e > \alpha$, with $\alpha$ as in
Lemma~\ref{lem:generate-small-q-patch}. Applying
Lemma~\ref{lem:generate-small-q-patch} yields a finite 
set $Y \subset \Ponepow$ such that any element
$m \in 1+q^{A-e}\Z_q/1+q^{A}\Z_q$ can be written as
$$m = \prod_{p \in Y} p^{\epsilon_{p}\delta_{p}} \mod q^{A},$$ with
$\epsilon_{p} \in \{0,1\}$ and $f(p^{\epsilon_{p}\delta_{p}}) = 1.$ 

Let $t$ be a sufficiently large prime so that
$(t,q)=1$ and $t \not \in Y$ holds. Then $t \equiv n \mod q^{A-e}$
for some $n \in N_{A}$, and consequently there exists
$m \in 1+q^{A-e}\Z_q/1+q^{A}\Z_q$ such that $tm \equiv n \mod q^{A}$.

In particular, $tm \in N_{A}$ and thus, chosing $\epsilon_{p} \in
\{0,1\}$ such 
that $m \equiv \prod_{p \in Y} p^{\epsilon_{p}\delta_{p}} \mod q^{A}$, we
arrive at
$$
0 = f(tm) = f(t) f(m)
= f(t)
\prod_{p \in Y : \epsilon_{p} = 1} f(p^{\delta_{p}})
=
f(t).
$$
This concludes the proof.
\end{proof}

\subsection{The case when $q$ is composite.} 
\label{sec:composite-q}
We now give a brief outline how to modify the argument for the
non-prime case.
Let $q = \prod_{i=1}^{k} q_{i}^{e_{i}}$ with $q_{i}$ prime.  Take $e$
large enough that $q_{1}^{e} > Q^{2}$.
Then the image of set (cf. the proof of
Proposition~\ref{prop:concluding-proof}) 
$$
\{r_{A} + mQ^{2} \}_{m < q^{A}/Q^{2}-1}
$$
under the reduction modulo $q_{1}^{Ae_{i} - e}  \prod_{i=2}^{k}
q_{i}^{Ae_{i} }$ map is onto.  The same argument as before 
(i.e., define $\alpha_{p,\delta}$ by $q_{1}^{\alpha_{p,\delta}} ||
p^{\delta \phi(q_{1}^{e_{1}})} - 1$ and arguing as before in the proof
of Proposition~\ref{accumulation})
yields a finite 
set $Y \subset \Ponepow$ such that any element
$m \in 1+q_{1}^{A-e}\Z_{q_{1}}/1+q_{1}^{A}\Z_{q_{1}}$ can be written as
$$m = \prod_{p \in Y} p^{\epsilon_{p}\delta_{p}} \mod q_{1}^{A},$$ with
$\epsilon_{p} \in \{0,1\}$ and $f(p^{\epsilon_{p}\delta_{p}}) = 1$. 
The same argument used to prove
Proposition~\ref{prop:concluding-proof} now applies.   

\begin{remark}\label{general}
Our proof gives somewhat stronger conclusion, namely if
$|\Pzero|=\infty$ then $|\Ponepow|<\infty.$ From here one can easily get a more refined information about the set of values of $\{f(p^n)\}_{n\ge 1}.$ We leave the details to the interested reader.
\end{remark}

\nocite{BCK}
\bibliographystyle{alpha} \bibliography{Automatabib}

\end{document}